\documentclass{elsarticle}

\usepackage[centertags]{amsmath}
\usepackage{amsfonts}
\usepackage{amssymb}
\usepackage{amsthm}
\usepackage{newlfont}

\newtheorem{theorem}{Theorem}

\newtheorem{lemma}{Lemma}
\newtheorem{proposition}{Proposition}
\newtheorem{corollary}{Corollary}

\newtheorem{definition}{Definition}

\newtheorem{remark}{Remark}

\newcommand{\inn}{\mathrm{ int}\,}

\newfont{\gothic}{eufm9 scaled 1200}

\begin{document}

\begin{frontmatter}
\title{A theorem of Piccard's type and its applications to polynomial functions and convex functions of higher orders}
\author{Eliza  Jab{\l}o{\'n}ska}
\ead{elizapie@prz.edu.pl}

\address{Department of Mathematics, Rzesz{\'o}w University of
Technology, Powsta\'{n}c\'{o}w~Warszawy~12, 35-959~Rzesz{\'o}w, POLAND}


\begin{abstract}
In the paper a theorem of Piccard's type is proved and, consequently, the continuity of $\mathcal{D}$-measurable polynomial functions of $n$-th order as well as $\mathcal{D}$-measurable $n$-convex functions is shown. The paper refers to the papers \cite{Gajda} and \cite{EJ}.
\end{abstract}
\begin{keyword}
Haar meager set, $\mathcal{D}$-measurable set, $\mathcal{D}$-measurable function, polynomial function, convex function
\MSC 54B30, 54E52, 39B52, 39B62
\end{keyword}
\end{frontmatter}

\section{Introduction}

In 2013 U.B.~Darji defined a $\sigma$-ideal of "small" sets in an abelian Polish group which is equivalent to the notation of meager sets in a locally compact group. He was motivated by J.P.R. Christensen's paper \cite{Ch} where the author defined \textit{Haar null} sets in an abelian Polish group  in such a way that in a locally compact group it is equivalent to the notation of Haar measure zero sets (see also \cite{Myc}, \cite{HSY}-\cite{HSY1}).

It is well known that there are some analogies between measure and category in locally compact groups (see \cite{Oxtoby}). It turns out (see \cite{Ch}, \cite{Darji} and \cite{EJ}) that in non-locally compact abelian Polish groups Haar meager sets have an analogous behavior to Haar null sets; e.g.
\begin{itemize}
\item both families are invariant $\sigma$-ideals,
\item $0\in\inn (A-A)$ for each universally measurable non-Haar null set as well as each Borel non-Haar meager set $A$,
\item each compact set is Haar null and Haar meager,
\item each set containing a translation of each compact set is non-Haar meager and non-Haar null.
\end{itemize}
Moreover, the notation of Christensen measurability introduced by P.~Fischer and Z. S{\l}odkowski \cite{FS} seems to be analogous to the notation of $\mathcal{D}$-measurability from \cite{EJ}; e.g.
\begin{itemize}
\item each Christensen measurable or $\mathcal{D}$-measurable homomorphism is continuous,
\item each Christensen measurable or $\mathcal{D}$-measurable $t$-Wright convex function is continuous
\end{itemize}
(see \cite{FS}, \cite{EJ}-\cite{JMI}, \cite{Olbrys}).

Here we prove a theorem of Piccard's type\footnote{We call it in this way because a consequence of this theorem is the fact that $0\in\inn  (A-A)$; cf. \cite[Theorem 2.9.1]{Kuczma}.} and next we show that each $\mathcal{D}$-measurable function which is $n$-convex with respect to the cone $S$ is continuous, as well as
each $\mathcal{D}$-measurable polynomial function of $n$-th order is continuous. Analogous results for Christensen measurable function has been obtained by Z. Gajda in \cite{Gajda}. 

\section{Preliminaries on $\mathcal{D}$-measurability}

The following fact will be useful in the sequel:

\begin{theorem}{\em \cite[p. 90]{DS}}
 If $X$ is an abelian Polish group, there exists an equivalent complete metric $\rho$ on $X,$ which is invariant; i.e. $\rho(x,y)=\rho(x+z,y+z)$ for every $x,y,z\in X.$
\end{theorem}

Let us recall some necessary definitions and theorems from \cite{Darji} and \cite{EJ}.

\begin{definition}
\em{\cite[Definition~2.1]{Darji}}
 Let $X$ be an abelian Polish group. A set $A\subset X$ is \textit{Haar meager} iff there is a Borel set $B\subset X$ with $A\subset B$, a compact metric space $K$ and a continuous function
 $f:K\to X$ such that $f^{-1}(B+x)$ is meager in $K$ for all $x\in X$. The family of all Haar meager sets in $X$ we denote by $\mathcal{HM}$.
\end{definition}

By the definition we can easily see that the family of all Haar meager sets is invariant, i.e. $A+x\in\mathcal{HM}$ for each $x\in X$ and $A\subset X$.

\begin{theorem}{\em \cite[Theorem~2.2, Theorem~2.9]{Darji}}\label{t1}
In each abelian Polish group the family $\mathcal{HM}$ is a~$\sigma$-ideal contained in the $\sigma$-ideal of all meager sets.
\end{theorem}

\begin{remark}
In \cite{Darji}  Darji proved that in a locally compact abelian Polish group the family of all Haar meager sets and the family of all meager sets are equivalent, but in each non-locally compact abelian Polish group it is not true.
\end{remark}

\begin{definition}\em{\cite[Definition~1]{EJ}}
Let $X$ be an abelian Polish group. A set $A\subset X$ is \textit{$\mathcal{D}$-measurable} iff $A=B\cup M$ for a Haar meager set $M\subset X$ and a Borel set
$B\subset X$.
\end{definition}

\begin{theorem}{\em \cite[Theorem~6]{EJ}}\label{t2}
The family $\mathcal{D}$ of all $\mathcal{D}$-measurable sets in each abelian Polish group is a~$\sigma$-algebra.
\end{theorem}

\begin{definition}{\em \cite[Definition~2]{EJ}}
Let $X$ be an abelian Polish group and $Y$ be a~topological space. A~mapping $f:X\to Y$ is a \textit{$\mathcal{D}$-measurable function} iff $f^{-1}(U)\in\mathcal{D}$ in $X$ for each open set $U\subset Y$.
\end{definition}

It is easy to observe the following

\begin{proposition}\label{jed}
Let $X$ be a real linear Polish space. If $A\subset X$ is Haar meager and $\alpha\in\mathbb{R}$ then $\alpha A$ is also Haar meager.
\end{proposition}

\begin{proof}
Let $\alpha\neq0$ (the case $\alpha=0$ is trivial). There exist a Borel set $B\supset A$, a compact metric space $K$ and a continuous function $f:K\to X$ such that $f^{-1}(x+B)$ is meager in $K$ for each $x\in X$. Define
$$g(k):=\alpha f(k)\;\;\mbox{for}\;\;k\in K.$$
Then, for each $y\in X$, we have
$$
g^{-1}(y+\alpha B)=g^{-1}\left(\alpha \left(\frac{y}{\alpha}+B\right)\right)=f^{-1}\left(\frac{y}{\alpha}+B\right).
$$
It means that the set $g^{-1}(y+\alpha A))$ is meager in $K$ and, consequently, $\alpha B$ is a~Borel Haar meager set. Thus $\alpha A\subset \alpha B$ is Haar meager.
\end{proof}

\section{On a theorem of Piccard's type}

Now we prove a lemma (cf. \cite[Lemma 15.5.1]{Kuczma}), which is a theorem of Piccard's type.

\begin{lemma}\label{B}
If $(X,+)$ is an abelian metric group and $A\subset X$ is a non-meager
set with the Baire property, then, for $n\in\mathbb{N}$,
$$
0\in \mathrm{int}\left \{x\in X: \bigcap _{k\in J_n} (A+kx) \mbox{\;is non-meager in\;} X\right\},
$$
where $J_n:=\{-n,\ldots,-1,0,1,\ldots,n\}$.
\end{lemma}

\begin{proof}
Let $A=\left(G\cup I_1\right)\setminus I_2,$ where $G$ is nonempty and open and $I_1, I_2$ are meager in $X$. Fix $g\in G$ and define $G_0:=G-g$, $A_0:=A-g$. Since $G_0$ is open and $0\in G_0$, so there exists $\varepsilon>0$ such that
$$U:=K(0,\varepsilon)\subset G_0.$$

Now, take any $x\in K\left(0,\frac{\varepsilon}{n+1}\right)$. Then $x\in U+kx$ (since $(1-k)x\in U$) for $k\in J_n$. Define
$$U_x:=\bigcap_{k\in J_n} (U+kx).$$
Clearly $U_x\neq\emptyset$ (because $x\in U_x$) and $U_x$ is open. Moreover,
$$
U_x\setminus (A_0+kx)\subset (G_0+kx)\setminus (A_0+kx)=(G_0\setminus A_0)+kx\subset I_2-g+kx
$$
for $k\in J_n$,
so $U_x\setminus (A_0+kx)$ is meager for such $k$. On the other hand
$$
\begin{array}{ll}
U_x &\subset\left[\bigcap_{k\in J_n} (A_0+kx)\right]\cup \left\{U_x\setminus \left[\bigcap_{k\in J_n} (A_0+kx)\right]\right\}\\[1.5ex]
& \subset\left[\bigcap_{k\in J_n} (A_0+kx)\right]\cup \bigcup_{k\in J_n}[U_x\setminus (A_0+kx)].
\end{array}$$
 Hence $\bigcap_{k\in J_n} (A_0+kx)$ is non-meager and so does the set  $$\bigcap_{k\in J_n} (A_0+g+kx)=\bigcap_{k\in J_n} (A+kx).$$

In this way we have proved that
$$K\left(0,\frac{\varepsilon}{n+1}\right) \subset\left\{x\in X:\bigcap_{k\in J_n}(A+kx) \mbox{\;is non-meager in\;} X\right\},$$
what ends the proof.
\end{proof}

Since in an abelian Polish group each Haar meager set is meager but the converse implication generally doesn't hold we would like to know if the above result can be generalized. In fact we prove an ana\-logous result to Gajda's theorem for non-zero Christensen measurable set.

\begin{theorem}\label{G}
Let $X$ be an abelian Polish group.
If $A\subset X$ is a Borel non-Haar meager set and $n\in\mathbb{N}$, then the set $$F_n(A):=\{x\in X: \bigcap_{k\in J_n} (A+kx)\not\in \mathcal{HM}\}$$ is a nonempty open set.
\end{theorem}

\begin{proof}
Clearly $0\in F_n(A)$, because $A\not\in\mathcal{HM}$.
For the indirect proof suppose that there exists a sequence $(x_i)_{i\in\mathbb{N}}\subset X$ such that
$$x_i\in \overline{K}(0,2^{-i})\setminus F_n(A)\;\;\mbox{for each}\;\;i\in\mathbb{N}.
 $$
 Now define $$A_0:=A\setminus \bigcup_{i=1}^\infty\left[\bigcap_{k\in J_n}(A+kx_i)\right].$$
 Since
 $x_i\not\in F_n(A)$ for each $i\in\mathbb{N},$ so $\bigcap_{k\in J_n}(A+kx_i)\in\mathcal{HM}$ for $i\in\mathbb{N}.$ Hence, by Theorem~\ref{t1}, $\bigcup_{i\in\mathbb{N}}\bigcap_{k\in J_n} (A+kx_i)\in\mathcal{HM}$ and, consequently, $A_0\not\in\mathcal{HM},$ because $A\not\in\mathcal{HM}.$

Since $A_0\not\in\mathcal{HM},$ for every compact metric space $K$ and continuous function $g:K\to X$ there is a $y_K\in X$ such that $g^{-1}(A_0+y_K)$ is non-meager in $K.$

Let $K:=\{0,1,\ldots,2n\}^{\aleph_0}$. It is well known that it is a compact abelian metric group with the operation $\oplus :K\times K\to K$ given by
$$
(k_i)_{i\in\mathbb{N}}\oplus(l_i)_{i\in\mathbb{N}}:=(k_i+_{2n+1} l_i)_{i\in\mathbb{N}}\;\;\mbox{for}\;\;(k_i)_{i\in\mathbb{N}},(l_i)_{i\in\mathbb{N}}\in K$$
(where $+_{2n+1}$ denotes the operation modulo $2n+1$)
and with the product metric
$$
d((k_i)_{i\in\mathbb{N}},(l_i)_{i\in\mathbb{N}}):=\sum_{i=1}^{\infty} 2^{-i}\overline{d}(k_i,l_i)\;\;\mbox{for}\;\;(k_i)_{i\in\mathbb{N}},(l_i)_{i\in\mathbb{N}}\in K, $$
where $\overline{d}$ is the discrete metric in $\{0,1,\ldots,2n\}.$ Define a function $g:K\to X$ as follows:
\begin{equation}\label{g}
g((k_i)_{i\in\mathbb{N}})=\sum_{i=1}^{\infty} a(k_i)x_i\;\;\mbox{for}\;\; (k_i)_{i\in\mathbb{N}}\in K,
\end{equation}
where $a:\{0,1,\ldots,2n\}\to J_n$ is given by
$$
a(m)=\left\{
\begin{array}{ll}
m,& m\in\{0,1,\ldots, n\},\\
m-(2n+1),& m\in\{n+1,\ldots,2n\}.
\end{array}
\right.
$$
To see that $g$ is well defined and continuous it is enough to prove that the series $\sum_{i=1}^{\infty} a(k_i)x_i$ is uniformly convergent. So, fix $\varepsilon>0.$ Then we can find a~positive integer $N$ such that $n2^{-N}<\varepsilon.$ Hence, for each $k>l\geq N$ we have:
$$
\begin{array}{ll}
\displaystyle\rho\left(\sum_{i=1}^k a(k_i)x_i,\sum_{i=1}^l a(k_i)x_i\right)&=\displaystyle \rho\left(0,\sum_{i=l+1}^k a(k_i)x_i\right)\leq\sum_{i=l+1}^k \rho\left(0,a(k_i)x_i\right)\\[1.5ex]
&\leq n \sum_{i=l+1}^k \rho\left(0,x_i\right)\leq\displaystyle n\sum_{i=l+1}^k 2^{-i}\\[1.5ex]
&\leq n \sum_{i=l+1}^\infty 2^{-i}=n2^{-l}\leq n2^{-N}< \varepsilon.
\end{array}
$$
Consequently, since $X$ is complete, the series $\sum_{i=1}^{\infty} a(k_i)x_i$ is uniformly convergent.

Now, for $K=\{0,1,\ldots,2n\}^{\aleph_0}$ and function $g$ given by \eqref{g} the set
 $$Z:=g^{-1}(A_0+y_K)$$ is not meager in $K$ for some $y_K\in X.$ Moreover, since $A$ is a Borel set, so does $A_0$ and $Z.$ Thus the set $Z$ is non-meager with the Baire property in $K$ and, in view of Lemma~\ref{B},
there is an open ball $K(0,2^{-t})$ with some $t\in\mathbb{N}$ such that
$$
K(0,2^{-t})\subset \{(k_i)_{i\in\mathbb{N}}\in K: \bigcap_{j\in J_n} \left(Z\oplus j(k_i)_{i\in\mathbb{N}}\right)\;\mbox{is not meager in}\;K\}.
$$
Let $s:=t+1$ and $e_{s}:=(0,\ldots,0,1,0,0,\ldots),$ where $1$ is in the $s$-th place. Since
$d(0,e_{s})=2^{-s}<2^{-t},$ so $e_{s}\in K(0,2^{-t}).$ Hence $$\bigcap_{j\in J_n} (Z\oplus je_{s})\neq\emptyset,$$ i.e. there is a $v:=(v_i)_{i\in\mathbb{N}}\in Z$ such that $v\ominus je_s\in Z$ for each $j\in J_n$. Thus
$$g(v\ominus je_s)\in A_0+y_K\;\;\mbox{for}\;\;j\in J_n.$$

Everyone can check that:
\begin{itemize}
\item if $v_s=0$ then
$$
g(v)-g(v\ominus je_s)=[-a(0-_{2n+1}j)]x_s=
jx_s\;\;\mbox{for}\;\; j\in J_n;$$
\item if $v_s\in\{1,\ldots, n\}$ then
$$\begin{array}{ll}
g(v)-g(v\ominus je_s)&=[v_s-a(v_s-_{2n+1}j)]x_s\\[1ex]
&=
\left\{\begin{array}{ll}
jx_s, & \;\; j\in\{-n+v_s,\ldots,n\},\\[1ex]
(j+2n+1)x_s, & \;\; j\in\{-n,\ldots,-n+v_s-1\};
\end{array}\right.
\end{array}
$$
\item if $v_s\in\{n+1,\ldots,2n\}$ then
$$
\begin{array}{ll}
g(v)-g(v\ominus je_s)&=[v_s-(2n+1)-a(v_s-_{2n+1}j)]x_s\\[1ex]
&=\left\{\begin{array}{ll}
jx_s, & j\in\{-n,\ldots,v_s-n-1\},\\[1ex]
(j-2n-1)x_s, & j\in\{v_s-n,\ldots,n\}.
\end{array}\right.
\end{array}
$$
\end{itemize}
Hence, if $v_s=0$ then $$g(v)=g(v\ominus je_s)+jx_s\;\;\mbox{for}\;\;j\in J_n$$ and, consequently,
$$g(v)\in \bigcap _{j\in J_n} (A_0+y_K+jx_s).$$
Next, if $v_s\in\{1,\ldots,n\}$ then
$$\begin{array}{l}
g(v)-v_sx_s=\\[1ex]
\;\;\;\;\;\;\;=\left\{
\begin{array}{ll}
g(v\ominus je_s)+(j-v_s)x_s,&j\in\{-n+v_s,\ldots,n\},\\[1ex]
g(v\ominus (j-2n-1)e_s)+(j-v_s)x_s,&j\in\{n+1,\ldots,n+v_s\},
\end{array}\right.\\[3 ex]
\;\;\;\;\;\;\;=\left\{
\begin{array}{ll}
g(v\ominus (k+v_s)e_s)+kx_s,& \;\;k\in\{-n,\ldots,n-v_s\},\\[1ex]
g(v\ominus (k+v_s-2n-1)e_s)+kx_s,& \;\;k\in\{n-v_s+1,\ldots,n\}
\end{array}\right.
\end{array}
$$
and hence
$$g(v)-v_sx_s\in \bigcap _{j\in J_n} (A_0+y_K+jx_s).$$
Finally, if $v_s\in\{n+1,\ldots,2n\}$ then
$$\begin{array}{l}
g(v)+(2n+1-v_s)x_s=\\[1ex]
\;\;\;\;\;\;\;=\left\{
\begin{array}{l}
g(v\ominus je_s)+(j+2n+1-v_s)x_s,\;\;\;\;\;\;\;\;j\in\{-n,\ldots,v_s-n-1\},\\[1ex]
g(v\ominus (j+2n+1)e_s)+(j+2n+1-v_s)x_s,\\[1ex]
\;\;\;\;\;\;\;\;\;\;\;\;\;\;\;\;\;\;\;\;\;\;\;\;\;\;\;\;\;\;\;\;\;\;\;\;\;\;\;\;\;\;\;\;\;\;\;\;\;\;\;\;\;\;\;\;\;\;\;j\in\{v_s-3n-1,\ldots,-n-1\}
\end{array}\right.\\[3 ex]
\;\;\;\;\;\;\;=\left\{
\begin{array}{l}
g(v\ominus (k-2n-1+v_s)e_s)+kx_s,\;\;\;\;\;\;\;k\in\{n+1-v_s,\ldots,n\},\\[1ex]
g(v\ominus (k+v_s)e_s)+kx_s,\;\;\;\;\;\;\;\;\;\;\;\;\;\;\;\;\;\;\;\;\;k\in\{-n,\ldots,n-v_s\}.
\end{array}\right.
\end{array}
$$
Thus
$$g(v)+(2n+1-v_s)x_s\in \bigcap _{j\in J_n} (A_0+y_K+jx_s).$$

It means that in all three cases there exists $z\in X$ such that
\begin{equation}\label{x7}
z\in \bigcap_{j\in J_n}(A_0+jx_s).
\end{equation}
Hence, $z\in A_0$ (for $j=0$) and, by the definition of $A_0$,
$z\in A$ and $z\not\in \bigcap_{k\in J_n} (A+kx_s)$. But $A_0\subset A$, so $z\not\in \bigcap_{k\in J_n} (A_0+kx_s)$, what contradicts \eqref{x7} and ends the proof.
\end{proof}

From the above theorem we obtain three interesting corollaries.

\begin{corollary}\label{c1}
For every $n\in\mathbb{N}$ and Borel set $A\subset X$ the set $F_n(A)$ is open (but not necessarily nonempty).
\end{corollary}

\begin{proof}
Suppose that the set $F_n(A)$ is nonempty and take any $x\in F_n(A)$. Then $$B_n:=\bigcap_{j\in J_n}(A+jx)$$ is a Borel non-Haar meager set. Hence, applying Theorem~\ref{G}, we have that $F_n(B_n)$ is a nonempty open set.

To end the proof it is enough to observe that $x+F_n(B_n)\subset F_n(A)$.
So, let $z\in x+F_n(B_n)$. Then $\bigcap_{j\in J_n} \left(B_n+j(z-x)\right)\not\in\mathcal{HM}$. We have to show
by induction that
\begin{equation}\label{w2}
\bigcap_{j\in J_n} \left(B_n+j(z-x)\right)\subset \bigcap_{j\in J_n} (A+jz)\;\;\mbox{for each}\;\;n\in\mathbb{N},
\end{equation}
because then $z\in F_n(A)$.

Clearly, for $n=1$ we have
$$\begin{array}{l}
B_1\cap (B_1+z-x)\cap (B_1-z+x)=\\[1ex]
\;\;\;\;\;\;\;\;\;\;\;\;\;\;\;=\left[(A+x)\cap A\cap (A-x)\right]\\[1ex]
\;\;\;\;\;\;\;\;\;\;\;\;\;\;\;\;\;\;\;\cap\; \left[(A+z)\cap (A+z-x)\cap (A+z-2x)\right]\\[1ex]
\;\;\;\;\;\;\;\;\;\;\;\;\;\;\;\;\;\;\;\cap\; \left[(A+2x-z)\cap (A-z+x)\cap (A-z)\right]\\[1ex]
\;\;\;\;\;\;\;\;\;\;\;\;\;\;\;\;\subset A\cap (A+z)\cap (A-z).
\end{array}
$$

Now assume that for some $n\geq 1$ condition \eqref{w2} holds. Then, since $B_{n+1}\subset B_n$ and $B_{n+1}\subset (A+(n+1)x)\cap (A-(n+1)x)$,
$$\begin{array}{l}
\bigcap_{j\in J_{n+1}} \left(B_{n+1}+j(z-x)\right)=\\[1ex]
\;\;\;\;\;\;\;\;\;\;\;\;\;\;=\bigcap_{j\in J_n} \left(B_{n+1}+j(z-x)\right)\\[1ex]
\;\;\;\;\;\;\;\;\;\;\;\;\;\;\;\;\;\;\cap\; \left(B_{n+1}+(n+1)(z-x)\right)\cap (B_{n+1}-(n+1)(z-x))\\[1ex]
\;\;\;\;\;\;\;\;\;\;\;\;\;\;\subset\left[\bigcap_{j\in J_n} \left(B_n+j(z-x)\right)\right]\cap (A+(n+1)z)\cap (A-(n+1)z)\\[1ex]
\;\;\;\;\;\;\;\;\;\;\;\;\;\;\subset\left[\bigcap_{j\in J_n} (A+jz))\right]\cap (A+(n+1)z)\cap (A-(n+1)z)\\[1ex]
\;\;\;\;\;\;\;\;\;\;\;\;\;\;=\bigcap_{j\in J_{n+1}} (A+jz),
\end{array}
$$
what ends the proof.
\end{proof}

\begin{corollary}\label{c2}
Let $X$ be an abelian Polish group. If $A\subset X$ is a $\mathcal{D}$-measurable set and $n\in\mathbb{N}$ then $F_n(A)$ is open (not necessarily nonempty).
If, moreover, $A\subset X$ is non-Haar meager $\mathcal{D}$-measurable set then $F_n(A)$ is nonempty.
\end{corollary}

\begin{proof}
Let $A=B\cup C$ with a Borel set $B$ and a Haar meager set $C$ in $X$. Using induction we prove that for each $n\in\mathbb{N}$ there is a Haar meager set $T_n$ such that
\begin{equation}\label{w1}
\bigcap_{j\in J_n} (A+jx)=\left[\bigcap_{j\in J_n} (B+jx)\right]\cup T_n.
\end{equation}

First take $n=1$. Then
$$(A+x)\cap (A-x)\cap A=[(B+x)\cap (B-x)\cap B]\cup T_1,$$
where
$$
\begin{array}{l}
T_1:=\;[C\cap (B+x)\cap (B-x)]\\[1ex]
\;\;\;\;\;\;\;\;\;\;\cup\;[(C-x)\cap (B+x)\cap (B\cup C)] \\[1ex]
\;\;\;\;\;\;\;\;\;\;\cup \;\{(C+x)\cap [(B-x)\cup (C-x)]\cap (B\cup C)\}\\[1ex]
\;\;\;\;\;\;\subset C\cup (C-x)\cup (C+x).
\end{array}$$
Since $C$ is Haar meager, by Theorem~\ref{t1} $T_1$ is also Haar meager.

Now assume that for some $n\geq 1$ condition \eqref{w1} holds with
a Haar meager set $T_n$. Then
$$\begin{array}{l}
\bigcap_{j\in J_{n+1}} (A+jx)=\\[1ex]
\;\;\;\;\;\;\;\;\;\;\;\;\;\;=\bigcap_{j\in J_n} (A+jx)\cap (A+(n+1)x)\cap (A-(n+1)x)\\[1ex]
\;\;\;\;\;\;\;\;\;\;\;\;\;\;=\left\{\left[\bigcap_{j\in J_n} (B+jx)\right]\cup T_n\right\}\cap (A+(n+1)x)\cap (A-(n+1)x)\\[1ex]
\;\;\;\;\;\;\;\;\;\;\;\;\;\;=\left[\bigcap_{j\in J_{n}} (B+jx)\cap (A+(n+1)x)\cap (A-(n+1)x)\right]\\[1ex]
\;\;\;\;\;\;\;\;\;\;\;\;\;\;\;\;\;\;\cup\; \left[T_n\cap (A+(n+1)x)\cap (A-(n+1)x)\right],\\[1ex]

\;\;\;\;\;\;\;\;\;\;\;\;\;\;=\left[\bigcap_{j\in J_{n}} (B+jx)\cap ((B\cup C)+(n+1)x)\cap ((B\cup C)-(n+1)x)\right]\\[1ex]
\;\;\;\;\;\;\;\;\;\;\;\;\;\;\;\;\;\;\cup\; \left[T_n\cap (A+(n+1)x)\cap (A-(n+1)x)\right],\\[1ex]

\;\;\;\;\;\;\;\;\;\;\;\;\;\;=\left[\bigcap_{j\in J_{n+1}} (B+jx)\right]\cup T_{n+1},
\end{array}
$$
where
$$
\begin{array}{ll}
T_{n+1}&:=\left[\bigcap_{j\in J_{n}} (B+jx)\right]\cap \left\{\left[(B+(n+1)x)\cap (C-(n+1)x)\right]\right.\\[1ex]
&\;\;\;\;\;\cup \left.\left[(C+(n+1)x)\cap ((B\cup C)-(n+1)x)\right]\right\}\\[1ex]
&\;\;\;\;\;\cup\left[T_n\cap (A+(n+1)x)\cap (A-(n+1)x)\right]\\[1ex]
&\;\subset (C-(n+1)x)\cup (C+(n+1)x)\cup T_n.
\end{array}
$$
Since $C, T_n$ are Haar meager, so does $T_{n+1}.$

Next, in view of \eqref{w1}, we have $F_n(A)=F_n(B)$ for each $n\in\mathbb{N}$ and, according to Corollary~\ref{c1}, we have the thesis.

If moreover $A$ is not Haar meager then  $0\in F_n(A)$,
 what ends the proof.
  \end{proof}

\begin{corollary}\label{c6}
Let $X$ be a real linear Polish space. If $A\subset X$ is a $\mathcal{D}$-measurable non-Haar meager set and $n\in\mathbb{N}$ then
$$0\in \inn \left\{x\in X: \bigcap_{j=1}^n \left[\left(A-\frac{x}{j}\right)\cap\left(A+\frac{x}{j}\right)\right]\neq \emptyset\right\}.$$
\end{corollary}

\begin{proof}
First observe that
\begin{equation}\label{w3}
\begin{array}{ll}
\bigcap_{k\in J_{n!}}\left(A+k\frac{x}{n!}\right)&
\subset \bigcap_{k=(n-1)!}^{n!}\left(A+k\frac{x}{n!}\right)\cap \left(A-k\frac{x}{n!}\right)\\[1ex]
&\subset \bigcap_{i=1}^n \left(A+\frac{x}{i}\right)\cap\left(A-\frac{x}{i}\right)
\end{array}\end{equation}
for $x\in X$, because $\left\{\frac{n!}{k}:\,k\in\{(n-1)!,\ldots,n!\}\right\}\supset\{1,\ldots,n\}$. For $x\in n!F_{n!}(A)$ the set
$\bigcap_{k\in J_{n!}}\left(A+k\frac{x}{n!}\right)$ is not Haar meager, so, in view of \eqref{w3}, $\bigcap_{i=1}^n \left(A+\frac{x}{i}\right)\cap\left(A-\frac{x}{i}\right)$ is not Haar meager, too. It means that
$$
\left\{x\in X:  \bigcap_{j=1}^n \left[\left(A-\frac{x}{j}\right)\cap\left(A+\frac{x}{j}\right)\right]\neq \emptyset\right\}\supset n!F_{n!}(A)
$$
and to end the proof it is
enough to apply Corollary~\ref{c2}.
\end{proof}

\section{Applications for $n$-convex functions and polynomial functions of $n$-th order}

Now, assume that $X$ is a real linear Polish space. From \cite{Gajda} we have the following useful definition.

\begin{definition}
Let $A\subset X$, $U\subset X$ and $U\neq\emptyset$. A point $x\in A$ is said to be $\mathbb{Q}_U-$internal point of $A$ iff for each $h\in U$ there is $\varepsilon >0$ such that $x+\lambda h\in A$ for each $\lambda\in (-\varepsilon,\varepsilon)\cap\mathbb{Q}$.
\end{definition}

The most important tool in proofs of main results is the following

\begin{lemma}\label{lem7}
Let $A\subset X$ be $\mathcal{D}-$measurable and assume that $U\subset X$ is not Haar meager. If $x\in A$ is a $\mathbb{Q}_U$-internal point of the set $A$, then $x\in\inn H_n(A)$ for each $n\in\mathbb{N}$, where
$$
H_n(A):=\{x\in X:\exists\; h\in X \;\mbox{such that}\; x-ih, x+ih\in A\;\mbox{for}\;i\in\{1,\ldots,n\}\}.
$$
\end{lemma}

The proof is analogous to the proof of Lemma~7 in \cite{Gajda} because of Corollary~\ref{c6} and Proposition~\ref{jed}, so we can omit it.\\

\begin{definition}
Let $D\subset X$ be a non-empty open convex set and $S\subset X$ be a~cone satisfying $S\cup (-S)\cup \{0\}=X$. Let $Y$ be a real normed space partially ordered by a relation $\preccurlyeq$ satisfying two conditions:
\begin{itemize}
\item if $x\preccurlyeq y$, then $x+z\preccurlyeq y+z$ and $\alpha x\preccurlyeq\alpha y$ for every $x,y,z\in Y$ and $\alpha\in [0,\infty)$;
\item if $0\preccurlyeq x\preccurlyeq y$, then $\|x\|\leq\|y\|$ for every $x,y\in Y$.
\end{itemize}
Let $n\in\mathbb{N}$. A function $f:D\to Y$ is called $n$-convex with respect to the cone $S$ iff
$$
\Delta_h^{n+1}f(x)=\sum_{j=0}^{n+1}(-1)^{n+1-j}\binom{n+1}{j}f(x+jh)\geq 0
$$
holds for all $x\in D$ and $h\in S$ such that $x+ih\in D$ for $i\in\{0,1,\ldots,n+1\}$.
\end{definition}

\begin{theorem}\label{uvw}
If $f:D\to Y$ is a $\mathcal{D}$-measurable function which is $n$-convex with respect to the cone $S$, then it is continuous.
\end{theorem}

\begin{proof}
In view of Theorem~\ref{t1}, Theorem~\ref{t2} and Proposition~\ref{jed} the proof of the theorem runs in the same way as the proof of Theorem~2 in \cite{Gajda}; it is enough to replace \cite[Lemma~7]{Gajda} by Lemma~\ref{lem7}.
 \end{proof}

\begin{definition}
Assume that $Y$ is a linear topological space (not necessary ordered and normed). Let $n\in\mathbb{N}$. A function $f:D\to Y$ is called a polynomial function of $n$-th order iff
$$
\Delta_h^{n+1}f(x)=0
$$
for all $x\in D$ and $h\in X$ such that $x+ih\in D$ for $i\in\{0,1,\ldots,n+1\}$.
\end{definition}

\begin{theorem}
If $f:D\to Y$ is a $\mathcal{D}$-measurable polynomial function of $n$-th order, then it is continuous.
\end{theorem}

\begin{proof}
 Since Theorem~\ref{t1}, Theorem~\ref{t2} and Proposition~\ref{jed} hold the proof of the theorem is analogous to the proof of Theorem~3 in \cite{Gajda}; we have to replace \cite[Lemma~7]{Gajda} by Lemma~\ref{lem7}.
 \end{proof}

\end{document}